\documentclass[12pt]{amsart}
\usepackage{amssymb}
\usepackage{amsmath}
\usepackage[all]{xy}
\usepackage{tikz}
\usetikzlibrary{positioning}
\usetikzlibrary{arrows.meta}
\usetikzlibrary{arrows}
\usepackage{amscd}
\usepackage{amsmath}
\input xy
\xyoption{all}

\newtheorem{lemma}{Lemma}[section]
\newtheorem{corollary}[lemma]{Corollary}
\newtheorem{theorem}[lemma]{Theorem}
\newtheorem{proposition}[lemma]{Proposition}

\theoremstyle{definition}

\newtheorem{definition}[lemma]{Definition}

\DeclareMathOperator{\Mod}{Mod}
\DeclareMathOperator{\modd}{mod}

\DeclareMathOperator{\Hom}{Hom}

\DeclareMathOperator{\brick}{brick}

\DeclareMathOperator{\sbrick}{sbrick}
\DeclareMathOperator{\fin-sbrick}{f-sbrick}

\DeclareMathOperator{\torf}{torf}

\DeclareMathOperator{\wide}{wide}
\DeclareMathOperator{\f-wide}{f-wide}
\DeclareMathOperator{\ff-wide}{ff-wide}
\DeclareMathOperator{\fg-tors}{fg-tors}
\DeclareMathOperator{\tors}{tors}
\DeclareMathOperator{\Tor}{Tor}
\DeclareMathOperator{\Has}{Hasse}
\DeclareMathOperator{\j-irr}{j-irr}
\DeclareMathOperator{\m-irr}{m-irr}
\DeclareMathOperator{\CJR}{CJR}
\DeclareMathOperator{\gen}{gen}
\DeclareMathOperator{\filt}{filt}
\DeclareMathOperator{\ssim}{sim}
\DeclareMathOperator{\ME}{ME}
\DeclareMathOperator{\MCE}{MCE}
\DeclareMathOperator{\I}{I}
\DeclareMathOperator{\II}{II}
\DeclareMathOperator{\mbrickc}{mbrick_{c.c.}}

\newtheorem{question}[lemma]{Question}
\newtheorem*{question 0*}{Question}
\newtheorem*{conjecture 0*}{Conjecture}
\newtheorem*{theorem 0*}{Theorem}
\newtheorem*{theorem a*}{Theorem A}
\newtheorem*{theorem b*}{Theorem B}
\newtheorem*{theorem c*}{Theorem C}
\newtheorem*{theorem d*}{Theorem D}
\newtheorem*{corollary e*}{Corollary E}
\newcounter{diagram}
\numberwithin{diagram}{section}

\begin{document}
	
	\title{Semibricks and Brick-finite algebras}

	\author{Alireza Nasr-Isfahani}
	\address{Department of Pure Mathematics\\
		Faculty of Mathematics and Statistics\\
		University of Isfahan\\
		P.O. Box: 81746-73441, Isfahan, Iran}
	\email{nasr$_{-}$a@sci.ui.ac.ir, nasr@ipm.ir}

	
\maketitle	
	{\centering\footnotesize Dedicated to Claus Michael Ringel on the occasion of his 80th birthday\par}	
	\begin{abstract}
Let $\Lambda$ be a finite-dimensional algebra. Using the kappa order on the lattice of torsion classes with canonical join representations, we obtain an equivalent condition for $\Lambda$ to be brick-finite. We show that $\Lambda$ is brick-finite if and only if every widely generated torsion class in $\modd \Lambda$ has finitely many covers with respect to the kappa order. Furthermore, we prove that every semibrick in $\modd \Lambda$ is a finite set if and only if every chain of wide subcategories of $\modd \Lambda$ is eventually constant. We also show that $\Lambda$ is brick-finite if and only if every chain of widely generated torsion classes of $\modd \Lambda$ is eventually constant. Finally, we show that $\Lambda$ is brick-finite if and only if every cofinally closed monobrick is a cofinal closure of some semibrick.

	\end{abstract}
	
	\maketitle

	%%%%%%%%%%%%%%%%%%%%%%%%%%%%%%%%%%%%%%

	\section{Introduction}
Let $\Lambda$ be a finite-dimensional $K$-algebra. A $\Lambda$-module $M$ is called \textit{brick} if its endomorphism ring is a division ring. Schur’s Lemma shows that this notion is a generalization of simple modules. Brick modules naturally appeared in the representation theory of finite-dimensional hereditary algebras (see \cite{ASS, R}). Recently, Enomoto in \cite{E1} showed that every simple object in a torsion-free class is a brick, and by using information on bricks, he classified torsion-free classes in the category of finitely generated $\Lambda$-modules. A pairwise Hom-orthogonal set of bricks is called \textit{semibrick}. Ringel in \cite{R} proved that there is a bijection between the semibricks $\mathcal{S}$ in $\modd \Lambda$ and the wide subcategories $\mathcal{W}$ of $\modd \Lambda$, where the subcategory $\mathcal{W}$ of $\modd \Lambda$ is called \textit{wide subcategory} if it is closed under kernels, cokernels, and extensions. This bijection sends a wide subcategory $\mathcal{W}$ of $\modd \Lambda$ to the semibrick $\mathcal{S}$ that is the set of the simple objects in $\mathcal{W}$.

Adachi, Iyama, and Reiten in \cite{AIR} introduced and investigated (support) $\tau$-tilting modules to provide a completion of the
class of tilting modules from the mutation point of view. A finite-dimensional $K$-algebra $\Lambda$ is called \textit{$\tau$-tilting finite} if
there are only finitely many isomorphism classes of basic $\tau$-tilting $\Lambda$-modules \cite{DIJ}. Demonet, Iyama, and Jasso in \cite{DIJ} proved that $\Lambda$ is $\tau$-tilting finite if and only if $\Lambda$ is \textit{brick-finite} (i.e., there are only finitely many isomorphism classes of bricks in $\modd \Lambda$). Brick-finite algebras are one of the active research directions in representation theory of algebras, and several authors studied this class of algebras \cite{AS, A, DIJ, DIRRH, E1}. Enomoto in \cite{E1} provides several equivalent conditions for brick-finite algebras. He proved that $\Lambda$ is brick-finite if and only if there are only finitely many semibricks in $\modd \Lambda$. A set $\mathcal{M}$ of bricks is called \textit{monobrick} if every non-zero map between bricks in $\mathcal{M}$ is an injection. Enomoto proved that $\Lambda$ is brick-finite if and only if every monobrick in $\modd \Lambda$ is a finite set if and only if every semibrick in $\modd \Lambda$ is a finite set, and the map $\mathrm{F}:\wide \Lambda\hookrightarrow \torf \Lambda$ is surjective, where $\wide \Lambda$ is the set of wide subcategories of $\modd \Lambda$, $\torf \Lambda$ is the set of torsion-free classes in $\modd \Lambda$ and for $\mathcal{W}\in \wide \Lambda$, $\mathrm{F}(\mathcal{W})$ is the smallest torsion-free class in $\modd \Lambda$ containing $\mathcal{W}$. He also posed the following conjecture.

\begin{conjecture 0*}(\cite[Conjecture 5.12]{E1})
Let $\Lambda$ be a finite-dimensional algebra. If every semibrick in $\modd \Lambda$ is a finite set, then $\modd \Lambda$ is brick-finite, that is, $\Lambda$ is $\tau$-tilting finite.
\end{conjecture 0*}

This conjecture remains open. Our primary goal in this paper is to investigate Enomoto's conjecture and brick-finite algebras, and to establish some new equivalent conditions for this class of algebras.

First, by using the kappa order on the lattice of torsion classes in $\modd \Lambda$ with canonical join representations (see Definitions \ref{D2}, \ref{D4} and \ref{D6}), we prove the following theorem, which will be useful throughout the rest of the paper.

Let $L$ be a poset and  $a, b\in L$. We say that $a$ \textit{covers} $b$ and denote it by
$a\gtrdot b$ if $a > b$ and there is no $x$ in $L$ satisfying $a > x > b$.
	
	\begin{theorem a*}(Theorem \ref{T3}) Let $\Lambda$ be a finite-dimensional algebra. Then $\Lambda$ is brick-finite if and only if the following conditions are satisfied.
\begin{itemize}
			\item[(I)] Any chain of wide subcategories of $\modd \Lambda$ becomes eventually constant.
\item[(II)] For any torsion class $\mathcal{T}\in(\tors \Lambda)_{\kappa}$ there are finitely many $\mathcal{T}_i\in(\tors \Lambda)_{\kappa}$ that $\mathcal{T}\lessdot_\kappa\mathcal{T}_i$.
\end{itemize}
	\end{theorem a*}

Note that we will show that condition II in the above theorem implies condition I. In fact, in Theorem D, we will show that $\Lambda$ is brick-finite if and only if it satisfies condition II.
	
Then we provide an equivalent condition for the existence of an infinite semibrick in $\modd \Lambda$. More precisely, we prove the following result.
	
		\begin{theorem b*}(Theorem \ref{T4}) There is an infinite semibrick in $\modd \Lambda$ if and only if there is an infinite chain $$W_1\subsetneq W_2\subsetneq W_3\subsetneq\cdots$$ of wide subcategories of $\modd \Lambda$.
	\end{theorem b*}
	
A torsion class $\mathcal{T}$ is called \textit{widely generated} if there is a wide subcategory $\mathcal{W}$ in $\modd \Lambda$ such that $\mathcal{T}$ is the smallest torsion class containing $\mathcal{W}$. By using the minimal (co-)extending modules defined and studied by Barnard, Carroll, and Zhu in \cite{BCZ}, we give a bijection between widely generated
torsion(-free) classes and semibricks in $\modd \Lambda$ (see Theorem \ref{T6}). Moreover, this bijection restricts to a bijection between finitely generated torsion(-free) classes and finite semibricks in $\modd \Lambda$ (see Corollary \ref{T11}). By using this result, we prove the following Theorem, which gives another equivalent condition for the existence of an infinite semibrick in $\modd \Lambda$.

		\begin{theorem c*}(Theorem \ref{T7}) There is an infinite semibrick in $\modd \Lambda$ if and only if there is a torsion class $\mathcal{T}$ and infinitely many torsion classes $\mathcal{T}_1, \mathcal{T}_2, \mathcal{T}_3, \ldots$ in $\modd \Lambda$ that for any $i\neq j$, $\mathcal{T}_i\neq \mathcal{T}_j$ and $\mathcal{T}_i$ covers $\mathcal{T}$ for each $i$.
	\end{theorem c*}

Now, by using these results, we can provide an equivalent condition for brick-finiteness.

\begin{theorem d*}(Theorem \ref{T8}) Let $\Lambda$ be a finite-dimensional $K$-algebra. Then the following are equivalent.
\begin{itemize}
\item[(1)] $\Lambda$ is brick-finite.
\item[(2)] For any torsion class $\mathcal{T}\in(\tors \Lambda)_{\kappa}$ there are finitely many $\mathcal{T}_i\in(\tors \Lambda)_{\kappa}$ that $\mathcal{T}\lessdot_\kappa\mathcal{T}_i$.
\item[(3)] There are finitely many $\mathcal{T}_i\in(\tors \Lambda)_{\kappa}$ that $0\lessdot_\kappa\mathcal{T}_i$.
\end{itemize}
	\end{theorem d*}

By \cite[Theorem 4.25]{E2}, the posets $(\tors \Lambda)_{\kappa}$ and $\wide \Lambda$ are isomorphic. Then, by the above theorem, if $\Lambda$ is brick-infinite, then $\wide \Lambda$ has an infinite antichain. However, whether it has an infinite chain is not clear. Theorem B shows that this is equivalent to Enomoto's conjecture. In contrast, the lattice $\tors \Lambda$ always has an infinite chain for brick-infinite algebras \cite[Proposition 3.9]{DIJ}, but whether it has an infinite antichain is not clear. Theorem C shows that this is also equivalent to Enomoto's conjecture.

By studying chains of widely generated torsion classes, we show that $\Lambda$ is brick-finite if and only if every chain widely generated torsion classes in $\modd \Lambda$ is eventually constant (see Theorem \ref{T12}). As a consequence, we provide the following equivalence condition for brick-finite algebras.

 	\begin{corollary e*}(Corollary \ref{T13}) Let $\Lambda$ be a finite-dimensional algebra. Then $\Lambda$ is brick-finite if and only if the following conditions are satisfied.
\begin{itemize}
\item[(1)] For any chain $\mathrm{T}(\mathcal{W}_1)\subseteq \mathrm{T}(\mathcal{W}_2)\subseteq \mathrm{T}(\mathcal{W}_3)\subseteq\cdots$ of widely generated torsion classes of $\modd \Lambda$, $\bigcup_{i=1}^{\infty}\mathrm{T}(\mathcal{W}_i)$ is a widely generated torsion class.
\item[(2)] Any semibrick in $\modd \Lambda$ is a finite set.
\end{itemize}
	\end{corollary e*}

From the above corollary, Enomoto's conjecture is equivalent to the claim that whenever every semibrick in $\modd \Lambda$ is a finite set, the union of any chain $\mathrm{T}(\mathcal{W}_1)\subseteq \mathrm{T}(\mathcal{W}_2)\subseteq \mathrm{T}(\mathcal{W}_3)\subseteq\cdots$ of widely generated torsion classes of $\modd \Lambda$ remains a widely generated torsion class.\\

Cofinally closed monobricks were introduced and investigated by Enomoto in \cite{E1}. Enomoto proved that there is a bijection between torsion-free classes and cofinally closed monobricks in $\modd \Lambda$. Finally, we give an injection from the set of semibricks
to the set of cofinally closed monobricks in $\modd \Lambda$ and show that $\Lambda$ is brick-finite if and only if this injection is a bijection if and only if every cofinally closed monobrick is a cofinal closure of some semibrick (for more details see Corollary \ref{C3} and the paragraph before it).

	%%%%%%%%%%%%%%%%%%%%%%%%%%%%%%%%%%%%%%
	
	\section{preliminaries}
	In this section, we recall the definitions and known results that we need throughout the paper.
	
Let $\Lambda$ be a finite-dimensional $K$-algebra and $\modd \Lambda$ be the category of finitely generated right $\Lambda$-modules. A subcategory $\mathcal{W}$ of $\modd \Lambda$ is called \textit{wide subcategory} if it is closed under extensions, kernels, and cokernels. We denote by $\wide \Lambda$ the set of wide subcategories of $\modd \Lambda$. Note that $\wide \Lambda$ forms a poset under inclusion.

A subcategory $\mathcal{T}$ of $\modd \Lambda$ is called \textit{torsion class} (resp. \textit{torsion-free class}) in $\modd \Lambda$ if it is closed under extensions and quotients (resp. extensions and subobjects) in $\modd \Lambda$. The set of torsion classes (resp. torsion-free classes) in $\modd \Lambda$ is denoted by $\tors \Lambda$ (resp. $\torf \Lambda$). $\tors \Lambda$ forms a poset under inclusion. $\tors \Lambda$ is a \textit{complete lattice}, that is, it has arbitrary joins and meets \cite{DIRRH}. This lattice had been the topic of many research \cite{AP, BCZ, BTZ, DIRRH, E2, E1}. Let $X$ be a collection of objects in $\modd \Lambda$, the smallest torsion class (resp. torsion-free class) in $\modd \Lambda$ containing $X$ is denoted by $\mathrm{T}(X)$ (resp.  $\mathrm{F}(X)$). There are the following standard dualities between $\tors \Lambda$ and $\torf \Lambda$. Given a torsion class $\mathcal{T}$, denote by $\mathcal{T}^\bot$ the set of modules
$M\in\modd \Lambda$ such that $\Hom_\Lambda(\mathcal{T} ,M) = 0$. The map $$(-)^{\bot} : \tors \Lambda\rightarrow \torf \Lambda$$ is a poset
anti-isomorphism with inverse given by
$$^\bot(-) : \torf \Lambda\rightarrow \tors \Lambda,$$ where $^\bot\mathcal{F}:=\{M\in\modd \Lambda|\Hom_\Lambda(M, \mathcal{F})=0\}.$

The characterization of functorially finite torsion classes using $\tau$-tilting theory is one of the first significant results in this direction. Let us recall the definition of a functorially finite subcategory.

	\begin{definition}\label{D1}
 	A full subcategory $\mathcal{B}$ of $\modd \Lambda$ is called {\em covariantly finite} in $\modd \Lambda$, if for every $M\in \modd \Lambda$ there exist an object $N\in\mathcal{B}$ and a morphism $f : M\rightarrow N$ such that, for all $N'\in\mathcal{B}$, the sequence of abelian groups $\Hom_\Lambda(N, N')\rightarrow \Hom_\Lambda(M, N')\rightarrow 0$ is exact. Such a morphism $f$ is called a {\em left $\mathcal{B}$-approximation of $M$}. The notions of {\em contravariantly
			finite subcategory} of $\modd \Lambda$ and {\em right $\mathcal{B}$-approximation} are defined dually. A {\em functorially
			finite subcategory} of $\modd \Lambda$ is a subcategory which is both covariantly and contravariantly finite
		    in $\modd \Lambda$.
	\end{definition}

For a given subcategory $\mathcal{B}$ of $\modd \Lambda$, we denote by $\gen(\mathcal{B})$ the subcategory of $\modd \Lambda$ containing all quotients of finite direct sums of objects
from $\mathcal{B}$ and $\filt(\mathcal{B})$ is the subcategory of all finitely generated right $\Lambda$-modules $M$ admitting
a finite filtration of the form
$$0 =M_0\subseteq M_1\subseteq M_2\subseteq\ldots\subseteq M_n = M$$
with $M_i/M_{i-1}\in \mathcal{B}$ for all $1\leq i\leq n$ and $n\in \mathbb{N}$.

Let $X$ be a collection of objects in $\modd \Lambda$. It is known that $\mathrm{T}(X)=\filt(\gen(X))$.

We denote a set of functorially finite wide subcategories of $\modd \Lambda$ by  $\f-wide \Lambda$ and a set of functorially finite wide subcategories $\mathcal{W}$ of $\modd \Lambda$ for which $\mathrm{T}(\mathcal{W})$ is a functorially finite torsion class, by $\ff-wide \Lambda$.

\begin{definition}\label{D8}
		\begin{itemize}
			\item[(1)] A $\Lambda$-module $M$ is called {\em brick}, if the endomorphism ring of $M$ is a division ring (i.e. every nonzero map $f:M\rightarrow M$ is an isomorphism). The set of all bricks in $\modd \Lambda$ is denoted by $\brick \Lambda$.
			\item[(2)] A set $\mathcal{S}$ of bricks in $\modd \Lambda$ is called a {\em semibrick} if $\Hom(S_1,S_2)=0$ for any $S_1\neq S_2\in \mathcal{S}$. The set of all semibricks in $\modd \Lambda$ is denoted by $\sbrick \Lambda$.
		\end{itemize}
	\end{definition}
	
For $\mathcal{W}\in\wide \Lambda$ we denote by $\ssim(\mathcal{W})$ the set of simple objects of $\mathcal{W}$. Ringel in \cite[1.2]{R} proved the following result.

\begin{theorem}\label{T5}$($\cite[1.2]{R}$)$ Let $\Lambda$ be a finite-dimensional algebra. Then we have a bijection between the set of wide subcategories and semibricks in $\modd \Lambda$:$$\xymatrix{
& \wide \Lambda  \ar@<1ex>[r]^{\ssim}
&\sbrick \Lambda. \ar@<1ex>[l]^{\filt}}$$
\end{theorem}

Angeleri H{\"u}gel and Sentieri in \cite[Questions 5.12.]{AS} posted the following question:

\begin{question}$($\cite[Questions 5.12.]{AS}$)$\label{Q1} The following is a list of necessary conditions which are satisfied when $\Lambda$ is a $\tau$-tilting finite algebra. Is any of them a sufficient condition?
\begin{itemize}
\item[(1)] Every wide subcategory of $\modd \Lambda$ is functorially finite.
\item[(2)] Every wide subcategory closed under coproducts of $\Mod \Lambda$ is closed under products.
\item[(3)] The target of any ring epimorphism $\Lambda\rightarrow\Gamma$ with $\Tor^{\Lambda}_1(\Gamma, \Gamma) =0$ is an artin algebra.
\end{itemize}
\end{question}	

Note that $($2$)$ implies $($1$)$ (See \cite[Corollary 5.7]{AS}). Also, $($2$)$ and $($3$)$ imply that $\Lambda$ is $\tau$-tilting finite by \cite[Corollary 5.11]{AS}.

Now, we show that if Enomoto's conjecture is true, then the conditions $1$ and $2$ in the question of Angeleri H{\"u}gel and Sentieri are sufficient.

\begin{proposition}\label{T10} Let $\Lambda$ be a finite-dimensional $K$-algebra which satisfies Enomoto's conjecture. Then the following are equivalent.
\begin{itemize}
\item[(1)] Every wide subcategory closed under coproducts of $\Mod \Lambda$ is closed under products.
\item[(2)] Every wide subcategory of $\modd \Lambda$ is functorially finite.
\item[(3)] $\Lambda$ is brick-finite.
\end{itemize}
\end{proposition}
\begin{proof} $($1$)$$\Rightarrow$$($2$)$ follows from \cite[Corollary 5.7]{AS}.

$($2$)$$\Rightarrow$$($3$)$ Assume that every wide subcategory of $\modd \Lambda$ is functorially finite. Then by \cite[Proposition 4.12]{E3}, every $\mathcal{W}\in \wide \Lambda$ is equivalent to $\modd \Gamma$ for some artin algebra $\Gamma$ and hence $\ssim(\mathcal{W})$ is a finite set. Therefore, by Theorem \ref{T5} every semibrick in $\modd \Lambda$ is a finite set and the result follows by Enomoto's conjecture.

$($3$)$$\Rightarrow$$($1$)$ follows from \cite[Corollary 5.11]{AS}.
\end{proof}

In the following proposition, we collect some results of \cite{DIJ, E1} and \cite{MS}, which give some equivalent conditions for brick-finite algebras that we need in this paper.

\begin{proposition}\label{P1}
		Let $\Lambda$ be a finite-dimensional $K$-algebra. Then the following are equivalent.
		\begin{itemize}
			\item[(1)] $\Lambda$ is brick-finite.
			\item[(2)] $\wide \Lambda$ is a finite set.
\item[(3)] $\f-wide \Lambda$ is a finite set.
\item[(4)] $\ff-wide \Lambda$ is a finite set.
		\end{itemize}
	\end{proposition}
\begin{proof} $(1)\Rightarrow(2)$ follows from \cite[Corollary 3.11]{MS} and Corollary 2.9, Theorems 3.8 and 4.2 of \cite{DIJ} (see also \cite[Theorem 5.5]{E1}).

$(2)\Rightarrow(3)$ and $(3)\Rightarrow(4)$ are obvious.

$(4)\Rightarrow(1)$ follows from \cite[Theorem 3.10]{MS} and Corollary 2.9 and Theorem 4.2 of \cite{DIJ}.
\end{proof}

Now we collect basic definitions and results on lattice theory.

Recall that a \textit{complete lattice} $L$ is a partially ordered set such that, for any subset $X,$ there is a
unique largest element of $L$ smaller than all elements of $X,$ the \textit{meet} of $X,$ denote by
$\bigwedge X,$ and a unique smallest element of $L$ larger than all elements of $X,$ the \textit{join} of
$X,$ denote by $\bigvee X.$

\begin{definition}\label{D2}$($\cite[Definitions 2.1, 2.2 and 4.1]{E2}$)$ Let $L$ be a complete lattice.
\begin{itemize}
\item[(1)] A \textit{join representation} of $l\in L$ is an expression of the form $l = \bigvee A$ for some subset $A\subseteq L$.
\item[(2)] Let $l = \bigvee A =\bigvee B$ be two join representations. $A$ is said \textit{refine} $B$ if for every
$a \in A$ there is some $b\in B$ with $a\leq b$.
\item[(3)] A join representation $l =\bigvee A$ is called a \textit{canonical join representation} if it satisfies the
following conditions:
\begin{itemize}
\item[(I)] $A$ refines every join representation of $l$.
\item[(II)] $A$ is an antichain, that is, $a_1\leq a_2$ with $a_1, a_2 \in A$ implies that $a_1 = a_2$.
\end{itemize}
\item[(4)] An element $l\in L$ is called \textit{completely join-irreducible} if $l =\bigvee X$ for some subset
$X \subseteq L$ implies $l\in X$. The set of completely join-irreducible elements of $L$ is denoted by $\j-irr^c L$.
\item[(5)] An element $m \in L$ is called \textit{completely meet-irreducible} if $m = \bigwedge X$ for some
subset $X \subseteq L$ implies $m \in X$. The set of completely meet-irreducible elements of $L$ is denoted by $\m-irr^c L$.
\item[(6)] For any $l\in L$, $l_*$ is defined as follows.
$$l_*:=\bigvee\{x\in L | x<l\}.$$
\end{itemize}
\end{definition}

\begin{definition}\label{D3}$($\cite[Definition 2.3]{E2}$)$
A lattice $L$ is called \textit{completely semidistributive} if it is a complete lattice and
satisfies the following conditions:
\begin{itemize}
			\item[(1)] For $a, b\in L$ and $X \subseteq L$, if $a\vee x = b$ for every $x\in X$, then $a\vee(\bigwedge X)=b$ holds.
			\item[(2)] For $a, b\in L$ and $X \subseteq L$, if $a\wedge x = b$ for every $x\in X$, then $a\wedge (\bigvee X) = b$ holds.
\end{itemize}
\end{definition}

By \cite[Theorem 3.1(a)]{DIRRH}, $\tors \Lambda$ is completely semidistributive.

Let $L$ be a poset. The \textit{Hasse quiver} of $L$ which is denoted by $\Has L$
is the quiver whose vertex set is $L$, and we draw an arrow $a\rightarrow b$ if $a > b$
and there is no $x$ in $L$ satisfying $a > x > b$. In this case, we say that $a$ \textit{covers} $b$ and denote it by
$a\gtrdot b$. The set of arrows in $\Has L$ denoted by $\Has_1 L$  and its element is called a \textit{Hasse
arrow} of $L$.

\begin{definition}\label{D4}$($\cite[Definitions 2.5 and 2.8]{E2}$)$ Let $L$ be a completely semidistributive lattice.
\begin{itemize}
			\item[(1)] The \textit{meet-irreducible labeling} $\mu: \Has_1 L \rightarrow
\m-irr^c L$ is defined as follows. For any $a\rightarrow b\in \Has_1L$, $\mu(a\rightarrow b):=max\{x\in L | a\wedge x = b\}$. Note that by \cite[Proposition 2.4]{E2} for any $a\rightarrow b\in \Has_1L$, $\mu(a\rightarrow b)$ exists (see also \cite{RST}).
\item[(2)] The kappa map $\kappa: \j-irr^c L\rightarrow
\m-irr^c L$ is defined as $\kappa(j) := \mu( j\rightarrow j_*) = \max\{x\in L | j\wedge x = j_*\}.$
\end{itemize}
\end{definition}

We denote by $\CJR(x) := A$ if $x$ has a canonical join representation $x =\bigvee A$. A subset
$L_0$ of $L$ is defined as $L_0 = \{x\in L | \CJR(x)\,\,\text{exists}\}$.

Barnard, Todorov, and Zhu in \cite[Definition 1.1.3]{BTZ} extend the kappa map as follows.

\begin{definition}\label{D5} $($\cite[Definition 1.1.3]{BTZ}, \cite[Definition 4.6]{E2}$)$ Let $L$ be a completely semidistributive lattice. Then the extended kappa map $\bar{\kappa}: L_0 \rightarrow L$ is defined as follows:
$\bar{\kappa}(l) =\bigwedge\{\kappa(j) | j\in \CJR(l)\}$.
\end{definition}

Enomoto in \cite{E2}, by using the extended kappa map of
Barnard, Todorov, and Zhu in \cite{BTZ} defined the kappa order as follows.

\begin{definition}\label{D6} $($\cite[Definition 4.24]{E2}$)$ Let $L$ be a completely semidistributive lattice. The \textit{kappa order} $\leq_{\kappa}$ is defined
on the set $L_0$ of elements with canonical join representations as follows.
$$a\leq_{\kappa} b \Leftrightarrow a\leq b\: \text{and} \:\kappa(a)\geq\kappa(b).$$
\end{definition}

The kappa order $\leq_{\kappa}$ gives a poset structure on $L_0$ and we denote by $L_{\kappa}$ the poset $(L_0,\leq_{\kappa})$. Enomoto in \cite{E2} proved the following theorem.

\begin{theorem}\label{T1} $($\cite[Theorem 4.25]{E2}$)$ Let $\Lambda$ be a finite-dimensional algebra. The map $T : \wide \Lambda\rightarrow\tors \Lambda$, where $T(\mathcal{W})$ is the smallest torsion class containing $\mathcal{W},$ induces a poset isomorphism $\wide \Lambda\simeq (\tors \Lambda)_{\kappa}$.
\end{theorem}

A torsion class $\mathcal{T}\in \tors \Lambda$ is called \textit{widely generated} if there is a wide subcategory $\mathcal{W}\in \wide \Lambda$ such that $\mathcal{T}=\mathrm{T}(\mathcal{W})$. Dually, a torsion-free class $\mathcal{F}\in \torf \Lambda$ is called \textit{widely generated} if there is a wide subcategory $\mathcal{W}\in \wide \Lambda$ such that $\mathcal{F}=\mathrm{F}(\mathcal{W})$. We denote the set of all widely generated torsion classes in $\modd \Lambda$ by $\tors_w \Lambda$ and the set of all widely generated torsion-free classes in $\modd \Lambda$ by $\torf_w \Lambda$. According to \cite[Theorem 3.25]{BH} a torsion class $\mathcal{T}\in \tors \Lambda$
has a canonical join representation if and only if it is widely generated. Therefore $(\tors \Lambda)_\kappa=\tors_w \Lambda$ as sets.

Now we recall the definition of minimal (co-)extending modules and some known results in this direction that we need throughout this paper.

\begin{definition}\label{D7} $($\cite[Definition 1.1]{BCZ}$)$
Let $\mathcal{T}$ be a torsion class in $\modd \Lambda$. $M\in \modd \Lambda$ is called a \textit{minimal extending module} for $\mathcal{T}$ if it satisfies the
following three properties:
\begin{itemize}
\item[(P1)] Every proper factor of $M$ is in $\mathcal{T}$;
\item[(P2)] If $0\rightarrow M\rightarrow X\rightarrow T\rightarrow 0$ is a non-split exact sequence with $T\in \mathcal{T}$, then
$X\in \mathcal{T}$;
\item[(P3)] $\Hom_\Lambda(\mathcal{T} ,M) = 0$.
\end{itemize}
The set of all minimal extending modules for a torsion class $\mathcal{T}$ is denoted by $\ME(\mathcal{T})$. Dually, for a torsion-free class $\mathcal{F}$ in $\modd \Lambda$ \textit{minimal co-extending modules} for $\mathcal{F}$ are defined. The set of all minimal co-extending modules for a torsion-free class $\mathcal{F}$ is denoted by $\MCE(\mathcal{F})$.
\end{definition}

Note that minimal extending modules for a torsion class $\mathcal{T}$ were
studied by several authors under the name torsion-free, almost torsion modules \cite{AS, S}.

\begin{lemma}\label{L1}$($\cite[Proposition 1.1]{EN}$($see also \cite[Remark 3.9 and Proposition 3.10]{AS}$))$
		Let $\mathcal{T}$ be a torsion class and $\mathcal{F}$ be a torsion-free class in $\modd \Lambda$.
		\begin{itemize}
			\item[(1)] Every minimal extending modules for a torsion class $\mathcal{T}$ is a brick.
\item[(2)] Every minimal co-extending modules for a torsion-free class $\mathcal{F}$ is a brick.
			\item[(3)] The set of isomorphism classes of all minimal extending modules for a torsion class $\mathcal{T}$ forms a semibrick.
\item[(4)] The set of isomorphism classes of all minimal co-extending modules for a torsion-free class $\mathcal{F}$ forms a semibrick.
\end{itemize}
\end{lemma}

\begin{theorem}\label{T2}$($\cite[Theorem 1.2]{BCZ}$($see also \cite{DIRRH}$))$
Let $\mathcal{T}$ be a torsion class in $\modd \Lambda$. Then the map
$$\eta_{\mathcal{T}} : M \longmapsto \filt(\mathcal{T}\cup\{M\})$$
is a bijection from the set $\ME(\mathcal{T})$ to the set of $\mathcal{T}'\in \tors \Lambda$ such that $\mathcal{T}\lessdot \mathcal{T}'$.
\end{theorem}

By using the standard duality  $(-)^{\bot} : \tors \Lambda\rightarrow \torf \Lambda$, one can relate
the upper covers of $\mathcal{T}\in \tors \Lambda$ to the lower covers of $\mathcal{T}^\bot$ in $\torf \Lambda$. Barnard, Carroll, and Zhu in \cite{BCZ} characterized covers of torsion-free classes by using minimal co-extending modules.

\begin{theorem}\label{T18}$($\cite[Theorem 2.11]{BCZ}$)$
Let $\mathcal{F}$ be a torsion-free class in $\modd \Lambda$. Then the map
$$\zeta_{\mathcal{F}} : M \longmapsto \filt(\mathcal{F}\cup\{M\})$$
is a bijection from the set $\MCE(\mathcal{F})$ to the set of $\mathcal{F}'\in \torf \Lambda$ such that $\mathcal{F}\lessdot \mathcal{F}'.$
\end{theorem}

According to \cite[Section 2.1]{R1}, a torsion class $\mathcal{T}\in \tors \Lambda$ is called \textit{finitely generated} if there is a $M\in \modd \Lambda$ with $\mathcal{T}=\mathrm{T}(M)$. Finitely generated torsion classes are also called compact torsion classes in the literature (see \cite{S}). Ringel in \cite[Addendum of Theorem 2.2]{R1} showed that any module $M\in \modd \Lambda$ has a factor module $N$ which is a semibrick and such that $\mathrm{T}(M)=\mathrm{T}(N)$.

\section{main results}
In this section, we prove our main results in this paper. First, we provide the following equivalent condition for brick-finite algebras, which has an important role in the proof of our main result.

\begin{theorem}\label{T3} Let $\Lambda$ be a finite-dimensional algebra. Then $\Lambda$ is brick-finite if and only if the following conditions are satisfied.
\begin{itemize}
			\item[(I)] Any chain of wide subcategories of $\modd \Lambda$ becomes eventually constant.
\item[(II)] For any torsion class $\mathcal{T}\in(\tors \Lambda)_{\kappa}$ there are finitely many $\mathcal{T}_i\in(\tors \Lambda)_{\kappa}$ that $\mathcal{T}\lessdot_\kappa\mathcal{T}_i$.
\end{itemize}
\end{theorem}
\begin{proof} Assume that $\Lambda$ is brick-finite. By Proposition \ref{P1} and Theorem \ref{T1}, the conditions $(\I)$ and $(\II)$ are follows. Now assume that the conditions $(\I)$ and $(\II)$ are satisfied. If $\Lambda$ is brick-infinite, then by Proposition \ref{P1} and Theorem \ref{T1} the poset $(\tors \Lambda)_{\kappa}$ is infinite. We claim that in this case we have an infinite chain in $(\tors \Lambda)_{\kappa}$. If every chain in $(\tors \Lambda)_{\kappa}$ is finite, then we have an infinite antichain $U_1, U_2, \ldots\in (\tors \Lambda)_{\kappa}$. But $0\subset U_i$ and since any chain in $(\tors \Lambda)_{\kappa}$ is finite, there exists some $X\in (\tors \Lambda)_{\kappa}$ with infinitely many covers in $(\tors \Lambda)_{\kappa}$, which contradicts condition $(\II)$. Therefore there is an infinite chain $\mathcal{T}_1\lessdot_\kappa\mathcal{T}_2\lessdot_\kappa\mathcal{T}_3\lessdot_\kappa\cdots\in (\tors \Lambda)_{\kappa}$ and our claim follows. By Theorem \ref{T1}, for each $i$ there exists $\mathcal{W}_i\in \wide \Lambda$ that $\mathcal{T}_i=T(\mathcal{W}_i)$.
By the proof of Theorem 4.25 of \cite{E2}, for $\mathcal{W}_1, \mathcal{W}_2\in \wide \Lambda$, $T(\mathcal{W}_1)\lneq_\kappa T(\mathcal{W}_2)$ if and only if $\mathcal{W}_1\subsetneq \mathcal{W}_2$. Then we have an infinite chain
$$\mathcal{W}_1\subsetneq \mathcal{W}_2\subsetneq \mathcal{W}_3\subsetneq\cdots$$ of wide subcategories of $\modd \Lambda$, which gives a contradiction to the assumption $(I)$ and the result follows.
\end{proof}
	
\begin{theorem}\label{T4} There is an infinite semibrick $\mathcal{S}\in \sbrick \Lambda$ if and only if there is an infinite chain $$\mathcal{W}_1\subsetneq \mathcal{W}_2\subsetneq \mathcal{W}_3\subsetneq\cdots$$ of wide subcategories of $\modd \Lambda$.
\end{theorem}
\begin{proof} Let $\mathcal{S}=\{S_1, S_2, S_3, \ldots\}$ be an infinite semibrick in $\modd \Lambda$. Then by Theorem \ref{T5} we have an infinite chain $$\filt(\{S_1\})\subsetneq\filt(\{S_1, S_2\})\subsetneq\filt(\{S_1, S_2, S_3\})\subsetneq\cdots$$
of wide subcategories in $\modd \Lambda$. Now assume that there is an infinite chain $$\mathcal{W}_1\subsetneq \mathcal{W}_2\subsetneq \mathcal{W}_3\subsetneq\cdots$$ of wide subcategories of $\modd \Lambda$. Then $\mathcal{W}=\bigcup_{i=1}^{\infty}\mathcal{W}_i$ is a wide subcategory of $\modd \Lambda$. By Theorem \ref{T5}, $\ssim(\mathcal{W})$ is a semibrick in $\modd \Lambda$. We claim that $\ssim(\mathcal{W})$ is an infinite semibrick. Assume on a contrary that $\ssim(\mathcal{W})=\{S_1, S_2, \ldots, S_n\}$ for some $n\in \mathbb{N}$. Then $\{S_1, S_2, \ldots, S_n\}\subseteq \mathcal{W}=\bigcup_{i=1}^{\infty}\mathcal{W}_i$ and there exists $j\in \mathbb{N}$ that $\{S_1, S_2, \ldots, S_n\}\subseteq \mathcal{W}_j$. Therefore $\filt(\{S_1, S_2, \ldots, S_n\})=\filt(\ssim(\mathcal{W}))\subseteq \mathcal{W}_j$. By Theorem \ref{T5}, $\filt(\ssim(\mathcal{W}))=\mathcal{W}$. Thus $\mathcal{W}=\bigcup_{i=1}^{\infty}\mathcal{W}_i\subseteq \mathcal{W}_j$ which gives a contradiction. Hence $\ssim(\mathcal{W})$ is an infinite semibrick and the result follows.
\end{proof}

We need
the following construction from \cite{IT}. Let $\mathcal{T}\in \tors \Lambda$ and
$$\alpha(\mathcal{T}) := \{X\in \mathcal{T} | \forall(g : Y\rightarrow X)\in \mathcal{T}, \ker(g)\in \mathcal{T}\}.$$
By \cite[Proposition 2.12]{IT}, $\alpha(\mathcal{T})\in \wide \Lambda$. Dually, $\beta(\mathcal{F})$ is defined for any $\mathcal{F}\in \torf \Lambda$.

\begin{proposition}\label{P2}
		Let $\mathcal{T}\in \tors \Lambda$, $\mathcal{F}\in \torf \Lambda$ and $\mathcal{S}\in \sbrick \Lambda$. Then the following statements hold.
		\begin{enumerate}
			\item $\alpha(\mathcal{T})=\filt(\MCE(\mathcal{T}^\bot))$.
\item $\beta(\mathcal{F})=\filt(\ME(^\bot\mathcal{F}))$.
\item $\mathcal{T}$ is widely generated if and only if $\mathcal{T}=\mathrm{T}(\MCE(\mathcal{T}^\bot))$.
\item $\mathcal{F}$ is widely generated if and only if $\mathcal{F}=\mathrm{F}(\ME(^\bot\mathcal{F}))$.
\item $\MCE(\mathrm{T}(\mathcal{S})^\bot)=\mathcal{S}.$
\item  $\ME(^\bot\mathrm{F}(\mathcal{S}))=\mathcal{S}.$
		\end{enumerate}
	\end{proposition}
\begin{proof}$($1$)$ By \cite[Proposition 3.10]{AS}, $\ssim(\alpha(\mathcal{T}))=\MCE(\mathcal{T}^\bot)$. Since $\alpha(\mathcal{T})\in \wide \Lambda$, then by Theorem \ref{T5}, $\alpha(\mathcal{T})=\filt(\ssim(\alpha(\mathcal{T})))=\filt(\MCE(\mathcal{T}^\bot))$.\

$($2$)$ By \cite[Proposition 3.10]{AS}, $\ssim(\beta(\mathcal{F}))=\ME(^\bot\mathcal{F})$. Since $\beta(\mathcal{F})\in \wide \Lambda$, then by Theorem \ref{T5}, $\beta(\mathcal{F})=\filt(\ssim(\beta(\mathcal{F})))=\filt(\ME(^\bot\mathcal{F}))$.\

$($3$)$ follows from $($1$)$, \cite[Theorem 7.2]{AP}, Lemma \ref{L1} and the fact that for any $\mathcal{S}\in \sbrick \Lambda$, $\mathrm{T}(\mathcal{S})=\mathrm{T}(\filt(\mathcal{S}))$.

$($4$)$ follows from $($2$)$,  \cite[Proposition 3.7]{EN}, Lemma \ref{L1} and the fact that for any $\mathcal{S}\in \sbrick \Lambda$, $\mathrm{F}(\mathcal{S})=\mathrm{F}(\filt(\mathcal{S}))$.

$($5$)$ Since $\mathrm{T}(\mathcal{S})=\mathrm{T}(\filt(\mathcal{S}))$ is a widely generated torsion class, then by \cite[Proposition 3.3]{MS}, $\filt(\mathcal{S})=\alpha(\mathrm{T}(\mathcal{S}))$. Therefore by \cite[Proposition 3.10]{AS}, $\ssim(\filt(\mathcal{S}))=\MCE(\mathrm{T}(\mathcal{S})^\bot)$. By Theorem \ref{T5}, $\ssim(\filt(\mathcal{S}))=\mathcal{S}$ and the result follows.

$($6$)$ By Theorem \ref{T5}, $\filt(\mathcal{S})\in \wide \Lambda$. Then by \cite[Proposition 3.7]{EN},  $\filt(\mathcal{S})=\beta(\mathrm{F}(\mathcal{S}))$. Therefore by \cite[Proposition 3.10]{AS}, $\ssim(\filt(\mathcal{S}))=\ME(^\bot\mathrm{F}(\mathcal{S}))$. By Theorem \ref{T5}, $\ssim(\filt(\mathcal{S}))=\mathcal{S}$ and the result follows.
\end{proof}

The subsequent result is requisite for the proof of the forthcoming Theorem, which possesses intrinsic significance.

\begin{theorem}\label{T6} We have the following bijections between widely generated torsion(-free) classes and semibricks in $\modd \Lambda$:

            \centering

\begin{tikzpicture}[node distance=1.5cm]

    \node (A) {$\tors_w \Lambda$};
    \node[right=of A] (B) {$\sbrick \Lambda$,};

    \draw[->, thick, black] ([yshift=1.5mm]A.east) -- node[above] {$\MCE(-^\bot)$} ([yshift=1.5mm]B.west);
    \draw[->, thick,black] ([yshift=-1.5mm]B.west) -- node[below] {$\mathrm{T}$} ([yshift=-1.5mm]A.east);
\end{tikzpicture}

            \centering

\begin{tikzpicture}[node distance=1.5cm]
    \node (A) {$\torf_w \Lambda$};
    \node[right=of A] (B) {$\sbrick \Lambda$.};

    \draw[->, thick, black] ([yshift=1.5mm]A.east) -- node[above] {$\ME(^\bot-)$} ([yshift=1.5mm]B.west);
    \draw[->, thick,black] ([yshift=-1.5mm]B.west) -- node[below] {$\mathrm{F}$} ([yshift=-1.5mm]A.east);
\end{tikzpicture}

\end{theorem}
\begin{proof} Let $\mathcal{T}\in \tors_w \Lambda$. Then by Lemma \ref{L1}, $\MCE(\mathcal{T}^\bot)\in \sbrick \Lambda$ and the map $\MCE(-^\bot)$ is well-defined. Now let $\mathcal{S}\in \sbrick \Lambda$. $\mathrm{T}(\mathcal{S})=\mathrm{T}(\filt(\mathcal{S}))$ and by Theorem \ref{T5}, $\filt(\mathcal{S})$ is a wide subcategory of $\modd \Lambda$. Therefore $\mathrm{T}(\mathcal{S})\in \tors_w \Lambda$ and the map $\mathrm{T}$ is well-defined. Proposition \ref{P2} shows that $\MCE(-^\bot)$ and $\mathrm{T}$ are mutually inverse to each other and the result follows. The proof of the second bijection is similar.
\end{proof}

We denote the set of all finitely generated torsion classes in $\modd \Lambda$ by $\fg-tors \Lambda$. Ringel in \cite[Theorem 2.2]{R1} proved that the map $\mathcal{S}=\{S_1, \ldots, S_n\}\mapsto \mathrm{T}(S_1\oplus\ldots\oplus S_n)$ provides a bijection between the set of finite semibricks and the set of finitely generated torsion
classes in $\modd \Lambda$. In the following result, we complete this picture.

We denote by $\fin-sbrick \Lambda$ the set of all finite semibricks in $\modd \Lambda$.

\begin{corollary}\label{T11} Let $\Lambda$ be a finite-dimensional $K$-algebra. Then the following diagram commutes, and the horizontal maps are bijections.

$$\begin{tikzpicture}[hookarrow/.style={{Hooks[right]}->},
node distance=1.5cm]
    \node (A) {$\tors_w \Lambda$};
    \node[right=of A] (B) {$\sbrick \Lambda$};
        \node[below=of A] (A') {$\fg-tors \Lambda$};
    \node[below=of B] (B') {$\fin-sbrick \Lambda.$};
    \draw[->, thick, black] ([yshift=1.5mm]A.east) -- node[above] {$\MCE(-^\bot)$} ([yshift=1.5mm]B.west);
    \draw[->, thick, black] ([yshift=-1.5mm]B.west) -- node[below] {$\mathrm{T}$} ([yshift=-1.5mm]A.east);
        \draw[->, thick, black] ([yshift=2mm]A'.east) -- node[above] {$\MCE(-^\bot)$} ([yshift=1.5mm]B'.west);
    \draw[->, thick, black] ([yshift=-1.5mm]B'.west) -- node[below] {$\mathrm{T}$} ([yshift=-1.mm]A'.east);
    \draw[hookarrow] (A')--(A);
      \draw[hookarrow] (B')--(B);
\end{tikzpicture}$$
\end{corollary}
\begin{proof} By \cite[Theorem 7.2]{AP} and \cite[Addendum of Theorem 2.2 and Theorem 2.5]{R1} any finitely generated torsion class is widely generated. Then the result follows from Theorem \ref{T6} and \cite[Theorem 2.2]{R1}.
\end{proof}

Corollary \ref{T11} shows that every semibrick in $\modd \Lambda$ is finite if and only if every widely generated torsion class in $\modd \Lambda$ is finitely generated.\\

Now we are ready to provide another equivalent condition for the existence of an infinite semibrick in $\modd \Lambda$.

\begin{theorem}\label{T7} There is an infinite semibrick $\mathcal{S}\in \sbrick \Lambda$ if and only if there is a torsion class $\mathcal{T}\in \tors \Lambda$ and infinitely many torsion classes $\mathcal{T}_1, \mathcal{T}_2, \mathcal{T}_3, \ldots$ in $\tors \Lambda$ that for any $i\neq j$, $\mathcal{T}_i\neq \mathcal{T}_j$ and $\mathcal{T}\lessdot \mathcal{T}_i$ for each $i$.
\end{theorem}
\begin{proof} Assume that there is a torsion class $\mathcal{T}\in \tors \Lambda$ and infinitely many torsion classes $\mathcal{T}_1, \mathcal{T}_2, \mathcal{T}_3, \ldots$ in $\tors \Lambda$ that for any $i\neq j$, $\mathcal{T}_i\neq \mathcal{T}_j$ and $\mathcal{T}\lessdot \mathcal{T}_i$ for each $i$. Then by Theorem \ref{T2} and Lemma \ref{L1}, $\ME(\mathcal{T})$ is an infinite semibrick and the result follows. Now assume that we have an infinite semibrick $\mathcal{S}\in \sbrick \Lambda$. By Theorem \ref{T6}, $\mathcal{S}=\ME(^\bot\mathrm{F}(\mathcal{S}))$. Therefore by Theorem \ref{T2}, there are infinitely many torsion classes $\mathcal{T}_1, \mathcal{T}_2, \mathcal{T}_3, \ldots$ in $\tors \Lambda$ that for any $i\neq j$, $\mathcal{T}_i\neq \mathcal{T}_j$ and $^\bot\mathrm{F}(\mathcal{S})\lessdot \mathcal{T}_i$ for each $i$ and the result follows.
\end{proof}

By the above theorem and Theorem \ref{T4}, we have the following corollary.

\begin{corollary}\label{T9} Let $\Lambda$ be a finite-dimensional $K$-algebra. Then the following are equivalent.
\begin{itemize}
\item[(1)] Every semibrick in $\modd \Lambda$ is a finite set.
\item[(2)] Any chain of wide subcategories of $\modd \Lambda$ becomes eventually constant.
\item[(3)] For any torsion class $\mathcal{T}\in(\tors \Lambda)$ there are finitely many $\mathcal{T}_i\in(\tors \Lambda)$ that $\mathcal{T}\lessdot\mathcal{T}_i$.
\end{itemize}
\end{corollary}	

Enomoto in \cite[Proposition 5.1]{E1} (see also \cite[Proposition 3.3]{MS}) proved that the map $\mathrm{F}:\wide \Lambda\hookrightarrow \torf \Lambda$, where $\mathrm{F}(\mathcal{W})$ is the smallest torsion-free class containing $\mathcal{W}$, is injective.
He also in \cite[Theorem 5.11]{E1} proved that $\Lambda$ is brick-finite if and only if every semibrick in $\modd \Lambda$ is a finite set, and the map $\mathrm{F}:\wide \Lambda\hookrightarrow \torf \Lambda$ is surjective.

\begin{corollary}\label{C2} Let $\Lambda$ be a finite-dimensional $K$-algebra. Then the following are equivalent.
\begin{itemize}
\item[(1)] $\Lambda$ is brick-finite.
\item[(2)] Every torsion-free class in $\modd \Lambda$ is widely generated.
\item[(3)] Every torsion class in $\modd \Lambda$ is widely generated.
\item[(4)] Every torsion class in $\modd \Lambda$ is finitely generated.
\end{itemize}
\end{corollary}	
\begin{proof} $($1$)$$\Leftrightarrow$$($2$)$ follows from \cite[Proposition 5.4]{AS}.

$($1$)$$\Rightarrow$$($3$)$ follows from \cite[Corollary 3.11]{MS}.

$($3$)$$\Rightarrow$$($1$)$ follows from \cite[Proposition 5.4]{AS}.

$($1$)$$\Rightarrow$$($4$)$ and $($4$)$$\Rightarrow$$($3$)$ follows from Corollary \ref{T11}.
\end{proof}

Now we can provide an equivalent condition for brick-finiteness.

\begin{theorem}\label{T8} Let $\Lambda$ be a finite-dimensional $K$-algebra. Then the following are equivalent.
\begin{itemize}
\item[(1)] $\Lambda$ is brick-finite.
\item[(2)] For any torsion class $\mathcal{T}\in(\tors \Lambda)_{\kappa}$ there are finitely many $\mathcal{T}_i\in(\tors \Lambda)_{\kappa}$ that $\mathcal{T}\lessdot_\kappa\mathcal{T}_i$.
\item[(3)] There are finitely many $\mathcal{T}_i\in(\tors \Lambda)_{\kappa}$ that $0\lessdot_\kappa\mathcal{T}_i$.  
\end{itemize}
\end{theorem}	
\begin{proof} $($1$)$$\Rightarrow$$($2$)$ follows from Theorem \ref{T3}.

$($2$)$$\Rightarrow$$($3$)$ is obvious.

$($3$)$$\Rightarrow$$($1$).$ For each brick $B\in \modd \Lambda$ there is a wide subcategory $\mathcal{W}_B=\filt(B)$ with only one simple object, and no proper non-zero subcategory of $\mathcal{W}_B$ is again wide in $\modd \Lambda$. Therefore for any brick $B$ we have $0\lessdot_\kappa\mathrm{T}(\mathcal{W}_B)$. If $\Lambda$ is brick-infinite, there are infinitely many bricks $B_i\in \modd \Lambda$ and infinitely many torsion classes $\mathcal{T}_i=\mathrm{T}(\mathcal{W}_{B_i})\in(\tors \Lambda)_{\kappa}$ that $0\lessdot_\kappa\mathcal{T}_i$, which gives a contradiction. Therefore $\Lambda$ is brick-finite and the result follows.
\end{proof}

\begin{theorem}\label{T12} Let $\Lambda$ be a finite-dimensional $K$-algebra. Then the following are equivalent.
\begin{itemize}
\item[(1)] $\Lambda$ is brick-finite.
\item[(2)] Any chain of torsion classes of $\modd \Lambda$ becomes eventually constant.
\item[(3)] Any chain of widely generated torsion classes of $\modd \Lambda$ becomes eventually constant.
\end{itemize}
\end{theorem}	
\begin{proof} $($1$)$$\Rightarrow$$($2$)$ follows from \cite[Theorem 3.8 and Lemma 3.10]{DIJ}.

$($2$)$$\Rightarrow$$($3$)$ is obvious.

$($3$)$$\Rightarrow$$($1$).$ If $\Lambda$ is not brick-finite, then by Corollary \ref{C2} there is a torsion class $\mathcal{T}$ which is not finitely generated. Let $M_1\in \mathcal{T}$, we have $\mathrm{T}(M_1)\subsetneq \mathcal{T}$. Now let  $M_2\in\mathcal{T}\setminus\mathrm{T}(M_1)$. Then we have $\mathrm{T}(M_1)\subsetneq \mathrm{T}(M_1, M_2)\subsetneq\mathcal{T}$. Doing this prossecs we can constract an infinite chain of finitely generated torsion classes of $\modd \Lambda$. Then by Corollary \ref{T11} we have an infinite chain of widely generated torsion classes of $\modd \Lambda$ which gives a contradiction. Therefore $\Lambda$ is brick-finite and the result follows.
\end{proof}

\begin{corollary}\label{T13} Let $\Lambda$ be a finite-dimensional $K$-algebra. Then $\Lambda$ is brick-finite if and only if the following conditions are satisfied.
\begin{itemize}
\item[(1)] For any chain $\mathrm{T}(\mathcal{W}_1)\subseteq \mathrm{T}(\mathcal{W}_2)\subseteq \mathrm{T}(\mathcal{W}_3)\subseteq\cdots$ of widely generated torsion classes of $\modd \Lambda$, $\bigcup_{i=1}^{\infty}\mathrm{T}(\mathcal{W}_i)$ is a widely generated torsion class.
\item[(2)] Any semibrick in $\modd \Lambda$ is a finite set.
\end{itemize}
\end{corollary}	
\begin{proof} Since the conditions 1 and 2 imply that any chain of widely generated torsion classes becomes eventually constant, the result follows from Theorem \ref{T12}.
\end{proof}

A set $\mathcal{M}$ of isomorphism classes of bricks in $\modd \Lambda$ is called \textit{monobrick} if every morphism between elements of $\mathcal{M}$ is either zero or an injection in $\modd \Lambda$. Enomoto in \cite{E1} introduced and studied monobricks.

A monobrick $\mathcal{N}$ is called a \textit{cofinal extension} of a monobrick $\mathcal{M}$ if $\mathcal{M}\subset \mathcal{N}$ and for any $N\in \mathcal{N}$, there exist $M\in \mathcal{M}$ and an injection $N\rightarrow M$ in $\modd \Lambda$. A monobrick $\mathcal{M}$ is called
\textit{cofinally closed} if there is no proper cofinal extension of $\mathcal{M}$. The union of all cofinal extensions of a monobrick $\mathcal{M}$, denoted by $\overline{\mathcal{M}}$, is called a \textit{cofinal closure} of $\mathcal{M}$. By \cite[Corollary 3.4]{E1}, $\overline{\mathcal{M}}$ is the unique cofinal extension of $\mathcal{M}$. The set of cofinally closed monobricks in $\modd \Lambda$ is denoted by $\mbrickc\Lambda$. Obviously, every semibrick is a monobrick. By using Theorem \ref{T6}, we show that there is an injective map from $\sbrick \Lambda$ to $\mbrickc\Lambda$.

According to \cite[Theorem 3.15]{E1}, we have the following bijection between the set of torsion-free classes and the set of cofinally closed monobricks in $\modd \Lambda$. $$\begin{tikzpicture}[node distance=1.5cm]
    \node (A) {$\torf\Lambda$};
    \node[right=of A] (B) {$\mbrickc\Lambda$.};

    \draw[->, thick, black] ([yshift=1.5mm]A.east) -- node[above] {$\ssim$} ([yshift=1.5mm]B.west);
    \draw[->, thick,black] ([yshift=-1.5mm]B.west) -- node[below] {$\filt$} ([yshift=-1.5mm]A.east);
\end{tikzpicture}$$

Therefore, by Theorem \ref{T6}, we have the following injection from the set of semibricks to the set of cofinally closed monobricks in $\modd \Lambda$. $$\begin{tikzpicture}[hookarrow/.style={{Hooks[right]}->},
node distance=1.2cm]
 \node (A) {$\Phi:\sbrick \Lambda$};
    \node[right=of A] (B) {$\torf_w \Lambda$};
    \node[right=of B] (C) {$\torf \Lambda$};
    \node[right=of C] (D) {$\mbrickc\Lambda$.};
    \draw[->, thick, black](A) -- node[above] {$\mathrm{F}$} (B);
     \draw[hookarrow] (B)--(C);
    \draw[->, thick,black] (C) -- node[above] {$\ssim$} (D);
\end{tikzpicture}$$

\begin{corollary}\label{C3} Let $\Lambda$ be a finite-dimensional $K$-algebra. Then the following are equivalent.
\begin{itemize}
\item[(1)] $\Lambda$ is brick-finite.
\item[(2)] The map $\Phi:\sbrick \Lambda\longrightarrow \mbrickc\Lambda$ is bijective.
\item[(3)] Every cofinally closed monobrick is a cofinal closure of some semibrick.
\end{itemize}
\end{corollary}
\begin{proof} $($1$)$$\Leftrightarrow$$($2$)$ follows from the above argument and Corollary \ref{C2}.

$($2$)$$\Leftrightarrow$$($3$)$ follows from the fact that for every semibrick $\mathcal{S}$, by \cite[Proposition 3.13]{E1}, $\overline{\mathcal{S}}=\ssim(\mathrm{F}(\mathcal{S}))$. 
\end{proof}

	\section*{acknowledgements}
This work is based upon research funded by Iran National Science Foundation (INSF) under project No. 4032107. The author is grateful to S. Asai, R. Kase, C. Paquette and F. Sentieri for
helpful discussions and insightful remarks.

\end{document}